\documentclass{amsart}

\usepackage{amsthm,amsmath,amsfonts}
\usepackage[all]{xy}
\usepackage{color}
\usepackage[pagebackref, colorlinks=true, linkcolor=blue, citecolor=blue]{hyperref}
\usepackage{tikz-cd}

\numberwithin{equation}{section}

\newtheorem{theorem}{Theorem}[section]
\newtheorem{lemma}[theorem]{Lemma}

\newtheorem{proposition}[theorem]{Proposition}

\newtheorem{corollary}[theorem]{Corollary}

\theoremstyle{definition}
\newtheorem{definition}[theorem]{Definition}
\newtheorem{example}[theorem]{Example}

\theoremstyle{remark}
\newtheorem{remark}[theorem]{Remark}

\newcommand{\sA}{\mathcal{A}}% calligraphic math
\newcommand{\sE}{\mathcal{E}}% calligraphic math
\newcommand{\sF}{\mathcal{F}}% calligraphic math
\newcommand{\sO}{\mathcal{O}}% calligraphic math
\newcommand{\Oh}{\mathcal{O}}% calligraphic math
\newcommand{\K}{\mathbb{K}}% blackboard math
\newcommand{\Z}{\mathbb{Z}}% blackboard math
\newcommand{\N}{\mathbb{N}}% blackboard math
\newcommand{\C}{\mathbb{C}}% blackboard math
\newcommand{\HH}{\mathbb{H}}% blackboard math
\newcommand{\de}{\partial}
\newcommand{\debar}{\overline{\partial}}

\newcommand{\Def}{\operatorname{Def}}

\newcommand{\MC}{\operatorname{MC}}

\newcommand{\End}{\operatorname{End}}
\newcommand{\Ext}{\operatorname{Ext}}
\newcommand{\EExt}{\mathbb{E}\operatorname{xt}}
\newcommand{\Hom}{\operatorname{Hom}}

\newcommand{\Spec}{\operatorname{Spec}}

\newcommand{\Id}{\operatorname{Id}}

\newcommand{\Tr}{\operatorname{Tr}}
\newcommand{\At}{\operatorname{At}}
\newcommand{\tr}{\operatorname{tr}}
\newcommand{\ch}{\operatorname{ch}}
\newcommand{\cs}{\operatorname{cs}}
\newcommand{\HOM}{\mathcal{H}om}

\newcommand{\treq}{\stackrel{\operatorname{tr}}{=}}

\title{$L_{\infty}$ liftings of semiregularity maps via Chern--Simons classes}

\author{Ruggero Bandiera}
\address{
Universit\`a degli studi di Roma La Sapienza,
Dipartimento di Matematica  Guido
Castelnuovo, P.le Aldo Moro 5,
I-00185 Roma, Italy.}
\email{bandiera@mat.uniroma1.it}

\author{Emma Lepri}
\address{School of Mathematics and Statistics,
University of Glasgow,
University Place,
Glasgow G12 8QQ, UK.}
\email{Emma.Lepri@glasgow.ac.uk}

\author{Marco Manetti}
\address{Universit\`a degli studi di Roma La Sapienza,
Dipartimento di Matematica  Guido
Castelnuovo,
P.le Aldo Moro 5,
I-00185 Roma, Italy.}
\email{marco.manetti@uniroma1.it}
\date{6 September, 2023}

\subjclass[2020]{14D15, 17B70}
\keywords{curved DG-algebras, Chern--Simons forms, Atiyah class, L-infinity maps, semiregularity}

\begin{document}

\maketitle

\begin{abstract} We introduce the notion of Chern--Simons classes for curved DG-pairs and we prove that a particular case of this general construction provides  canonical $L_\infty$ liftings of Buchweitz--Flenner semiregularity maps 
for coherent sheaves on complex manifolds.
\end{abstract}

\section{Introduction}

In this short introduction we only try to explain some of the motivations underlying this paper and we refer to Section~\ref{sec.curvedpairs} for a  presentation of the main results.

Semiregularity maps have a long and interesting history starting with Severi \cite{severi}, who called a curve $C$ in a surface $S$  semiregular  if the natural  map $H^0(S,\Omega^2_S)\to H^0(C,\Omega^2_S|_C)$ is surjective and proved, in modern terminology, that the Hilbert scheme of $S$ is smooth at every semiregular curve. 
The analogous result for smooth hypersurfaces in a compact complex manifold was proved by Kodaira and Spencer 
\cite[p. 482]{kodaspen}. Nowadays the above results are usually expressed in terms of semiregularity maps: 
if $Z$ is a (possibly singular) hypersurface in a smooth manifold $X$, its semiregularity map is the morphism
\[ \sigma\colon H^1(Z,N_{Z|X})\to H^2(X,\Oh_X)\]
induced by the short exact sequence $0\to \Oh_X\to \Oh_X(Z)\to N_{Z|X}\to 0$.  
Clearly $Z$ is semiregular if and only if the semiregularity map is injective and it is easy to extend Kodaira--Spencer's theorem to the following more general statement: \emph{the obstructions to embedded deformations of $Z$ in $X$ are contained in the kernel of $\sigma$}, see e.g. \cite[Thm. 8.1.5]{LMDT}.

When $Z$ is a locally  complete intersection of codimension $p$ in a smooth projective variety $X$, Bloch 
was able to define a semiregularity map $\sigma\colon H^1(Z,N_{Z|X})\to H^{p+1}(X,\Omega_X^{p-1})$ and prove, by using variations of Hodge structures, that every simple obstruction to embedded deformations of $Z$ in $X$ is annihilated by $\sigma$ \cite{bloch}.
By general deformation theory, see e.g. \cite[Example 11.0.1]{Ha10}, the cohomology group $H^1(Z,N_{Z|X})$ is an obstruction space for the Hilbert functor of the closed immersion $Z\subset X$. This means that for every small extension of finitely generated local Artin $\C$-algebras $0\to \C\to A\to B\to 0$
and every embedded deformation of $Z$ over $\Spec(B)$, there exists a canonical obstruction  
$u\in H^1(Z,N_{Z|X})$, which vanishes if and only if the deformation lifts to $\Spec(A)$. 

The obstruction $u$ as above is called \emph{simple} if the differential map $d\colon \C\to \Omega_{A/\C}\otimes_A B$ is injective. It is plain that the notion of simple obstruction makes sense for every functor of Artin rings admitting an obstruction theory. 

Although for general deformation problems the simple obstructions are a proper class of all obstructions 
\cite{Man,LMDT}, in characteristic 0 their vanishing is sufficient to ensure smoothness and hence, if the semiregularity map is injective, then $Z$ has unobstructed embedded deformations in $X$.  

Bloch's paper leaves unanswered the question of whether every obstruction belongs to the kernel of the semiregularity map: a positive answer to this question was given in \cite{ManettiSemireg} for smooth subvarieties of compact K\"{a}hler manifolds and in \cite{semireg} for globally complete intersection subvarieties.

A major breakthrough in the theory of semiregularity maps was given  by Buchweitz and Flenner in the paper \cite{BF}. 
They consider, as a special case of a more general construction, for every coherent sheaf $\sF$ on a complex manifold $X$ the sequence of maps 
\[ \tau_k\colon \Ext^2_X(\sF,\sF)\to H^{k+2}(X,\Omega_X^{k}),\quad
\tau_k(x)=\frac{(-1)^k}{k!}\Tr(\At(\sF)^kx),\qquad k\ge 0,\]
where $\At(\sF)\in \Ext^1_X(\sF,\sF\otimes\Omega^1_X)$ is the Atiyah class of $\sF$,
called \emph{semiregularity maps} of $\sF$; this name is motivated by the fact,  proved in \cite{BF}, that when $\sF=\Oh_Z$ with $Z$ a locally complete intersection of codimension $p$, then Bloch's semiregularity map $\sigma$ is the composition of $\tau_{p-1}$ with the natural 
map $H^1(Z,N_{Z|X})\to \Ext^2_X(\Oh_Z,\Oh_Z)$. 

The space $\Ext^2_X(\sF,\sF)$ has a natural interpretation as an obstruction space for deformations of $\sF$ and, in the same paper, Buchweitz and Flenner proved that if $X$ is projective then every semiregularity map annihilates simple obstructions. As regards the annihilation  of all obstructions, the natural strategy suggested in \cite{BF} is to interpret each semiregularity map as the obstruction map of a morphism of deformation theories with unobstructed target, and this paper goes exactly in this direction.

Among the possible formal frameworks for the category of deformation problems we use here the one, valid over  fields of characteristic 0, asserting that every deformation problem is controlled by a homotopy class of differential graded Lie algebras and a morphism of deformation theories is a morphism in the derived category of DG-Lie algebras, or equivalently an $L_{\infty}$ morphism.  The above strategy
 is easy to employ for the map $\tau_0$ (see for example \cite{DMcoppie}) and 
we recover a classical theorem by Mukai and Artamkin \cite{Arta}. 
An $L_{\infty}$ lifting of $\tau_1$ depending on the choice of a connection in a resolution of $\sF$  is constructed by explicit computation in \cite{linfsemireg}.

In this paper we extend the construction of \cite{linfsemireg} to every $\tau_k$; here the explicit computation (too difficult to manage for $k>1$) is replaced by an algebraic argument involving Chern--Simons classes of curved DG-algebras.

In conclusion, we reach our initial goal and prove, in particular, the following result:

\begin{corollary}[=Corollary~\ref{cor.main3}] Let $\sF$ be a coherent sheaf on a complex projective manifold  $X$. 
Then for every $k\ge 0$ the semiregularity map 
\[ \tau_k\colon \Ext^2_X(\sF,\sF)\to H^{k+2}(X,\Omega^{k}_X)\]
annihilates all obstructions to deformations of $\sF$.
\end{corollary}

According to results of Buchweitz and Flenner, the above corollary immediately implies: 

\begin{corollary} Let $Z$ be a codimension $p$ locally complete intersection subvariety of a smooth complex projective manifold 
$X$. Then every obstruction to embedded deformations of $Z$ is contained in the kernel of Bloch's semiregularity map 
$H^1(Z,N_{Z|X})\to H^{p+1}(X,\Omega_X^{p-1})$.
\end{corollary}

It is worthwhile and essential to mention the unpublished paper by Pridham~\cite{Pri}, where a similar result is proved in the framework of derived algebraic geometry (instead of  $L_{\infty}$ algebras) and using
Goodwillie's theorem (instead of Chern--Simons classes).

We refer to \cite{LiePairs} for another application of our main result; namely, we extend the definition of  semiregularity maps for  modules over Lie agebroids on separated schemes of finite type and we prove that 
they annihilates all the obstroctions to deformations.

\bigskip
\section{Curved DG-pairs and statement of the main results}
\label{sec.curvedpairs}

By a graded algebra, we mean a graded associative algebra with unit over a fixed field $\K$ of characteristic $0$.  Every graded associative algebra is also a graded Lie algebra, with the bracket given by 
the graded commutator  $[a,b]=ab-(-1)^{|a||b|}ba$.
For every vector subspace $E\subset A$ we shall denote: by $E^{(k)}$, $k\ge 1$, the linear span of all the products $e_1\cdots e_k$, with $e_i\in E$ for every $i$, and  by 
$EA$ the linear span of all the products $ea$, with  $e\in E$  and $a\in A$.

\begin{definition}[\cite{Posi}]\label{def.curved-algebra} A curved DG-algebra is the datum $(A,d,\cdot,R)$ of a graded associative unitary algebra $(A,\cdot)$ together with a degree one derivation $d\colon A^*\to A^{*+1}$ and a degree two element $R\in A^2$, called \emph{curvature}, such that
	\[ d(R)=0,\qquad d^2(x)=[R,x]=R\cdot x-x\cdot R\quad\,\,\forall\, x\in A. \]
\end{definition}

For notational simplicity we shall write $(A,d,R)$ in place of $(A,d,\cdot,R)$ when the product $\cdot$ is clear from the context. We denote by $[A,A]\subset A$ the \emph{linear span} of all the graded commutators $[a,b]=ab-(-1)^{|a||b|}ba$. Following \cite{GS} we 
call $A/[A,A]$ the \emph{cyclic space} of $A$ and we 
denote by $\tr\colon A\to  A/[A,A]$ the quotient map.  Notice that $[A,A]$ is a homogeneous Lie ideal
and then the cyclic space inherits  a natural structure of DG-Lie algebra with trivial bracket.

\begin{definition}\label{def.curved-ideal} Let $A=(A,d,R)$ be a curved DG-algebra. 
A \emph{curved Lie ideal} in $A$ is a homogeneous Lie ideal 
$I\subset A$  such that $d(I)\subset I$ and $R\in I$. 

By a \emph{curved DG-pair} we mean the data $(A,I)$ of a curved DG-algebra $A$ equipped with a
curved Lie ideal $I$. 
\end{definition}

In particular, for every curved DG-pair $(A,I)$, the quotient $A/I$ is a (noncurved) DG-Lie algebra, the  projection map $\pi\colon A\to A/I$ is a morphism of graded Lie algebras and, for every $k\ge 1$, the subset $I^{(k)}A$ is an associative bilateral ideal of $A$.

\begin{example}
It is useful to briefly anticipate from Section~\ref{sec.connection} the following paradigmatic geometric example of curved DG-pair. Let $E$ be a holomorphic vector bundle on a complex manifold $X$ equipped with a connection compatible with the holomorphic structure and denote by 
$R\in A_X^{1,1}(\End(E))\oplus A_X^{2,0}(\End(E))$ the curvature. Denoting by 
$d$ the induced connection on the associated bundle  $\End(E)$, we have that 
$(A_X^{*,*}(\End(E)),d,R)$ is a curved DG-algebra and $I=A_X^{>0,*}(\End(E))$ is a curved Lie ideal. In this case the DG-Lie algebra  $A/I=A_X^{0,*}(\End(E))$ is the Dolbeault resolution of $\End(E)$ and controls the deformations of the vector bundle $E$. Notice that $I$ is also an associative ideal and $I^{(k)}A=I^{(k)}=A_X^{\ge k,*}(\End(E))$ for every $k>0$.
\end{example}

The classical theory of Atiyah classes and the above example suggest the introduction of the following objects associated with a curved DG-pair. 

\begin{definition}\label{def.atiyahclass} Let $A=(A,d,R)$ be a curved DG-algebra and 
$I\subset A$ a curved  Lie ideal. 
The \emph{Atiyah cocycle} of the pair $(A,I)$ is  the class of $R$ in the DG-vector space 
$\dfrac{I+I^{(2)}A}{I^{(2)}A}$. The \emph{Atiyah class} of the pair $(A,I)$ is the cohomology 
class of the  Atiyah cocycle:
\[ \At(A,I)=[R]\in H^2\left(\dfrac{I+I^{(2)}A}{I^{(2)}A}\right).\]
\end{definition}

\begin{definition}\label{def:semiregmap} Let $A=(A,d,R)$ be a curved DG-algebra and 
$I\subset A$ a curved  Lie ideal. 
For every integer $k\ge 0$, we introduce 
the morphism of complexes of vector spaces 
\[ \sigma^k_1\colon \frac{A}{I}\to \frac{A}{[A,A]+I^{(k+1)}A}[2k],\qquad \sigma^k_1(x)=\frac{1}{k!}\tr(R^{{k}}x).\]
\end{definition}

Notice that $\sigma^k_1$ depends only on the Atiyah cocycle of the pair $(A,I)$, while the induced map in cohomology 
\[ \sigma^k_1\colon H^*\left(\frac{A}{I}\right)\to H^{2k+*}\left(\frac{A}{[A,A]+I^{(k+1)}A}\right),\qquad \sigma^k_1(x)=\frac{1}{k!}\tr(\At(A,I)^{{k}}x),\]
depends only on the Atiyah class. 

The \emph{semiregularity map} of the pair $(A,I)$ is defined as the degree 2 component  
\[ \sigma^k_1\colon H^2\left(\frac{A}{I}\right)\to H^{2k+2}\left(\frac{A}{[A,A]+I^{(k+1)}A}\right),\qquad \sigma^k_1(x)=\frac{1}{k!}\tr(\At(A,I)^{{k}}x),\]
of the above map.

\begin{remark} The name semiregularity map is clearly motivated by the analogous definition by Buchweitz and Flenner
\cite{BF}. More precisely, 
every morphism $\sigma_1^k$  factors as the composition of two morphisms of differential graded vector spaces
\[ \frac{A}{I}\xrightarrow{\;\tau^k_1\;}\frac{I^{(k)}A+[A,A]}{[A,A]+I^{(k+1)}A}\hookrightarrow
\frac{A}{[A,A]+I^{(k+1)}A},\qquad \tau^k_1(x)=\frac{1}{k!}\tr(R^{{k}}x),\]
and the direct generalisation of Buchweitz--Flenner's semiregularity maps should be the maps
induced by $\tau^k_1$ in the group $H^2(A/I)$, up to signs. However, 
several geometric considerations about Abel--Jacobi maps  (see e.g. the introduction of \cite{semireg}) strongly suggest  that, from the point of view of deformation theory, the right objects to consider are the maps $\sigma^k_1$.
\end{remark}

\begin{remark}
	For every $x\in A^1$ one can consider the twisted derivation $d_x=d+[x,-]$ and an easy computation shows that 
$A_x:=(A,d_x,R_x)$ remains a curved DG-algebra with curvature $R_x=R+d(x)+\frac{1}{2}[x,x]$. In particular, if $x\in I$ then 
$I$ is a curved Lie ideal also for $A_x$, the derivations $d,d_x$ induce the same differential in $A/I$ and $A/[A,A]$, the difference
$R_x-R$ is exact in $A/[A,A]$ and therefore the semiregularity maps of the pairs $(A,I)$ and $(A_x,I)$ induce the same map in cohomology.

\end{remark}

Since  $[A,A]+I^{(k+1)}A$ is a Lie ideal,  the space 
$\dfrac{A}{[A,A]+I^{(k+1)}A}[2k]$ inherits from $A$ a structure of DG-Lie algebra with trivial bracket and it is obvious that $\sigma^0_1$ is a morphism of DG-Lie algebras. 
It is easy to see that in general $\sigma^k_1$ is not a morphism of DG-Lie algebras for  
$k>0$. It is therefore natural to ask whether $\sigma^k_1$ is the linear component of an $L_{\infty}$ morphism.

In Section~\ref{sec.lifting} we prove the following result.

\begin{theorem}[=Corollary~\ref{cor.semireg}]\label{thm.main1} Let $I$ be a curved Lie ideal of a curved DG-algebra $(A,d,R)$ and let $\pi$ denote the projection $A\to A/I$.  
For every $k\ge 0$, one can associate, in a canonical fashion, an $L_{\infty}$ morphism 
\[ \sigma^k\colon \frac{A}{I}\rightsquigarrow \dfrac{A}{[A,A]+I^{(k+1)}A}[2k]\] having the map $\sigma^k_1$ as its linear component with every choice of morphism of graded vector spaces  $s\colon A/I\to A$ such that $\pi s=\Id_{A/I}$.
\end{theorem}

%%%% VECCHIO:

%\red{\begin{theorem}[=Corollary~\ref{cor.semireg}]\label{thm.main1} Let $I$ be a curved Lie ideal of a curved DG-algebra $(A,d,R)$ and denote by $\pi\colon A\to A/I$ the projection.  
%	Then to every $k\ge 0$ and every morphism of graded vector spaces  $s\colon A/I\to A$ such that $\pi s=\Id_{A/I}$ it is canonically associated an $L_{\infty}$ morphism 
%	\[ \sigma^k\colon \frac{A}{I}\rightsquigarrow \dfrac{A}{[A,A]+I^{(k+1)}A}[2k]\]
%	with linear component the  map $\sigma^k_1$.
%\end{theorem}}

The proof of the above theorem is constructive and an explicit description of the higher components 
of $\sigma^k$ is possible but rather cumbersome for general sections $s$. 
In Section~\ref{sec.explicit} we study the higher components of the 
$L_{\infty}$ morphism  of  Theorem~\ref{thm.main1} under the additional assumption that $s$ is a morphism of graded Lie algebras. This hypothesis is satisfied in most of the  
applications and has the effect of a dramatic simplification of the algebraic and combinatorial 
aspects.

\begin{definition}[\cite{Quillen88}]\label{def.tracemap} A \emph{trace map} on a   curved DG-algebra $(A,d,R)$ is the data of a complex of vector 
spaces $(C,\delta)$ and a morphism of graded vector spaces $\Tr\colon A\to C$ such that $\Tr\circ  d=\delta\circ \Tr$ and $\Tr([A,A])=0$.
\end{definition}
  
Thus every trace map $\Tr\colon A\to C$ factors to a morphism of abelian DG-Lie algebras 
$A/[A,A]\to C$ and we have the following immediate consequence of the above theorem.

\begin{corollary}\label{cor.main1} Let $I$ be a curved Lie ideal of a curved DG-algebra $(A,d,R)$ and let $\Tr\colon A\to C$ be a trace map. 
Then for every $k\ge 0$ there exists  an $L_{\infty}$ morphism 
\[ \eta^k\colon \frac{A}{I}\rightsquigarrow \dfrac{C}{\Tr(I^{(k+1)}A)}[2k]\]
with linear component 
\[ \eta^k_1\colon \frac{A}{I}\to \dfrac{C}{\Tr(I^{(k+1)}A)}[2k],\qquad \eta^k_1(x)=\frac{1}{k!}\Tr(R^{{k}}x).\]
\end{corollary}

By general facts, see e.g. \cite{ManRendiconti,Man} and references therein, we have a clear application of the above results to deformation theory. 

In the situation of Corollary~\ref{cor.main1}, denote for simplicity by $C_k$ the quotient complex $C_k:=C/\Tr(I^{(k+1)}A)$, and  suppose  that a given deformation problem is controlled by the DG-Lie algebra $A/I$. Then the 
$L_{\infty}$ morphism $\eta^k$ induces a morphism of deformation functors
\[\eta^k\colon \Def_{A/I}\to \Def_{C_k[2k]}\]
that at the level of tangent and obstruction spaces gives the maps
\[ H^i(A/I)\to H^{2k+i}(C_k),\qquad x\mapsto \frac{1}{k!}\Tr(\At(A,I)^k x),\qquad i=1,2.\]
Since $C_k[2k]$ is abelian, the deformation functor $\Def_{C_k[2k]}$ is unobstructed and therefore the above map 
$H^2(A/I)\to H^{2k+2}(C_k)$ annihilates the obstructions.

\bigskip
\section{Curved DG-algebras and Chern--Simons classes}
\label{sec.chernsimons}

The general theory of Chern--Simons classes for differential graded associative algebras \cite{Quillen88} extends naturally to the curved case.

Let $(A,d,R)$ be a curved DG-algebra.  Then for every $x\in A^1$ we have  the twisted curved DG-algebra $A_x:=(A,d_x,R_x)$, where $d_x(a)=d(a)+[x,a]$ and $R_x=R+d(x)+x^2=R+d(x)+\dfrac{1}{2}[x,x]$.

Let $t$ be a central indeterminate of degree $0$, and consider the family of polynomials 
\[ P(t)^k_x=\sum_{i=1}^{k}R_{tx}^{i-1}xR_{tx}^{k-i}=\sum_{i=1}^{k}(R + t d(x) + t^2 x^2)^{i-1}x(R + t d(x) + t^2 x^2)^{k-i}
\in A[t],\]
with $k\ge 0$ an integer and $x\in A^1$.

\begin{lemma}\label{lem.traccia-diff-Wk} 
	In the above notation, for every $k\ge 0$ and every $x\in A^1$ we have
	\[ R_x^k-R^k=d\left(\int_0^1 P(t)^k_xdt\right)+\left[x,\int_0^1 tP(t)^k_xdt\right].\]
\end{lemma}

\begin{proof} In the graded algebra $A[t]$ consider the derivations $\de_t=\dfrac{d~}{dt}$ and 
	$d_{tx}=d+[tx,-]$.
	Since $R_x^k-R^k=\int_0^1 \de_t(R_{tx}^k) dt$ it is sufficient to prove that 
	\begin{equation}\label{eqlem21} \de_t(R_{tx}^k)=d(P(t)^k_x)+[x,tP(t)^k_x]=d_{tx}(P(t)^k_x).\end{equation}
	Since 	
	$d^2(x)=[R,x]$, $d(x^2)=\frac{1}{2}d[x,x]=[d(x),x]$, $[x^2,x]=0$, we have 
	\[ d_{tx}(R_{tx})= d(R + td(x)+t^2x^2 )+[tx,R + td(x)+t^2x^2 ] = td^2(x) +t^2d(x^2) -t[R+td(x),x]=0 \]
	and 
	\[ \de_t(R_{tx}) =\de_t(R+td(x)+t^2x^2) =d(x)+2tx^2=d_{tx}(x).\]
	By the Leibniz formula, for every $k\ge 0$  we have
	\[
	d_{tx}(P(t)^k_x)=\sum_{i=1}^kd_{tx}\Big( R_{tx}^{i-1}xR_{tx}^{k-i}\Big)=\sum_{i=1}^kR_{tx}^{i-1}d_{tx}(x)R_{tx}^{k-i}=
	\sum_{i=1}^kR_{tx}^{i-1}\de_t(R_{tx})R_{tx}^{k-i}=\de_t(R_{tx}^{k}).\]
\end{proof}

Denote by  $\tr \colon A \to A/{[A,A]}$ the projection; this is the universal trace of $A$ in the sense that every trace map $A\to C$ is induced from $\tr$ by a unique morphism of DG-vector spaces $A/[A,A]\to C$.
For notational simplicity  we denote by $a \treq b$ the fact that $\tr(a)= \tr(b)$.

Following the theory of Chern classes, we can define the (universal) Chern character 
\[ \ch(A)=\sum_{k\ge 0}\ch(A)_k,\qquad \ch(A)_k\in H^{2k}\left(\frac{A}{[A,A]}\right),\]
where $\ch(A)_k$ is the cohomology class of $\dfrac{1}{k!}\tr(R^k)$.

Similarly, following Chern--Simons' theory \cite{CS,GS,Quillen88},  it also makes sense to define  the  
(universal) Chern--Simons class 
\[ \begin{split}
&\cs=\sum_{k>0}\cs_{2k-1},\qquad 
\cs_{2k-1}\colon A^1\to (A/[A,A])^{2k-1},\\[4pt]
&\cs_{2k-1}(x)=\frac{1}{(k-1)!}\tr \int_0^1 R_{tx}^{k-1}x\,dt \in \left(\frac{A}{[A,A]}\right)^{2k-1},\qquad k\ge 1,\; x\in A^1\,,\end{split}\]
where as before $R_{tx}:=R+td(x)+t^2x^2$, $t\in \K$, $x\in A^1$, denotes the curvature of the twisted curved DG-algebra $A_{tx}:=(A,d_{tx},R_{tx})$.

\begin{lemma}\label{lemma.traccia-diff-Wk}
	The Chern character is invariant under twisting: more precisely for every  $x\in A^1$ and  every $k\ge 1$ we have 
	\[  d(\cs_{2k-1}(x))=\frac{1}{k!}\tr \big( R_x^k - R^k \big).\]
\end{lemma}

\begin{proof} Immediate consequence of Lemma~\ref{lem.traccia-diff-Wk} since 
	\[R_{tx}^{k-1}x\treq \frac{1}{k}\sum_{i=1}^{k}R_{tx}^{i-1}xR_{tx}^{k-i}\]
	and therefore
	\[ \cs_{2k-1}(x)=\frac{1}{k!}\tr  \int_0^1 P(t)^k_xdt.\]
\end{proof}

For the explicit computations in the following Section \ref{sec.explicit}, it will be useful to introduce the elements 
\begin{equation}\label{equ.Wk}
W(x)^{k+1} = \frac{1}{k!} \int_{0}^{1} R_{tx}^{k}x\, dt = \frac{1}{k!} \int_{0}^{1} (R + t d(x) + t^2 x^2)^{k}x\, dt\ \ \in A^{2k+1},\quad x\in A^1,\;k\ge 0,
\end{equation}
as a representative set of liftings to $A$ of  Chern--Simons classes.

\bigskip
\section{Convolution algebras and $L_{\infty}$ liftings of $\sigma_1^k$}
\label{sec.lifting}

Our next step is  to prove that  curved DG-algebras are preserved by taking convolution with the bar construction of a DG-Lie algebra.

For a graded vector space $V$ we shall denote by $V[1]$ the same vector space with the degrees shifted by $-1$. More precisely, if $v\in V$ is homogenous of degree $|v|$, then the degree of $v$ in $V[1]$ is $|v|-1$. Unless otherwise specified, for any $v\in V[1]$ we shall denote by $|v|$ the degree of $v$ as an element of $V$.

In this paper we adopt the following sign convention for the  d\'ecalage isomorphisms: given a pair of graded vector spaces $V,W$,  for every $i>0$, $k\in \Z$ we consider the isomorphisms:  
\begin{equation}\label{equ.decalage} 
\begin{split}&\mbox{d\'ec}\colon  \Hom^{k}_{\K}(V^{\wedge i},W)\to 
\Hom^{k+i-1}_{\K}(V[1]^{\odot i},W[1]),\\
&\mbox{d\'ec}(f)(v_1,\ldots,v_i)=(-1)^{k+i-1+\sum_{s=1}^i(i-s)(|v_s|-1)}f(v_1,\ldots,v_i).\end{split}
\end{equation}
In particular, for $i=1$ and $k=0$ the d\'ecalage isomorphism is the identity.

Let $(L,\debar,[-,-])$ be a DG-Lie algebra. Then 
there exists a counital DG-coalgebra structure on the symmetric coalgebra $S(L[1])$, where the differential $Q\colon S(L[1])\to S(L[1])$ is given in Taylor coefficients $q_i\colon L[1]^{\odot i}\to L[1]$ by 
\[q_1(x)=-\debar(x),\qquad q_2(x,y)=(-1)^{|x|}[x,y],\qquad q_i=0 \,\,\,\mbox{for $i\not=1,2$},\]
where $|x|$ denotes  the degree of $x$ in $L$.
In other words $q_1$ and $q_2$ are the images of $\debar$ and $[-,-]$ under the d\'ecalage isomorphisms \eqref{equ.decalage}.  

More precisely, see e.g. \cite{LadaMarkl,LMDT}, $Q$ decomposes as $Q=Q_0+Q_1$, where $Q_0,Q_1:S(L[1])\to S(L[1])$ are the coderivations defined by $Q_0(1)=Q_1(1)=0$ and
\[ Q_0(x_1\odot\cdots\odot x_n)=\sum_{i=1}^n (-1)^{i+|x_1|+\cdots+|x_{i-1}|}x_1\odot\cdots \odot\debar(x_i)\odot\cdots\odot x_n,\]
\[ Q_1(x_1\odot\cdots\odot x_n)=\sum_{\tau\in S(2,n-2)} \varepsilon(\tau)(-1)^{|x_{\tau(1)}|} [x_{\tau(1)},x_{\tau(2)}]\odot x_{\tau(3)}\odot\cdots\odot x_{\tau(n)},\] for every $x_1,\ldots,x_n\in L[1]$, $n\ge1$, 
where  we denote by $S(i,n-i)$ the set of $(i,n-i)$-unshuffles, i.e., permutations $\tau\in S_{n}$ such that $\tau(1)<\cdots<\tau(i)$ and $\tau(i+1)<\cdots<\tau(n)$, and by $\varepsilon(\tau)$ the symmetric Koszul sign defined by the identity $x_{\tau(1)}\odot \cdots\odot x_{\tau(n)}=\varepsilon(\tau)\,x_1\odot\cdots\odot x_{n}$ in the symmetric power $L[1]^{\odot n}$.

In particular, for every $x,y\in L[1]$ we have 
\[ Q(x)=-\debar(x),\qquad Q(x\odot y)=-\debar(x)\odot y-(-1)^{|x|-1}x\odot \debar(y)+(-1)^{|x|}[x,y].\] 
Notice also that $Q_1(L[1])=0$, $Q_0(L[1]^{\odot i})\subset L[1]^{\odot i}$ and $Q_1(L[1]^{\odot i})\subset L[1]^{\odot i-1}$ for every $i$. 

Given a curved DG-algebra $(A,d,R)$, we introduce the following notations:
\[\mathbf{C}(L,A)_i:=\Hom^*_{\K}\Big(L[1]^{\odot i}, A\Big),\qquad\mathbf{C}(L,A)=\bigoplus_{i\ge0}\mathbf{C}(L,A)_i\subset \Hom^*_{\K}\Big(S(L[1]),A\Big).\]
	The unshuffle coproduct $\Delta\colon S(L[1])\to S(L[1])^{\otimes 2}$ and the algebra product 
	$m\colon A^{\otimes 2}\to A$ induce an associative product $f\star g:=m(f\otimes g)\Delta$ on the space $\mathbf{C}(L,A)$, called the \emph{convolution product}.  More explicitly, if $f\in\mathbf{C}(L,A)_i$ and $g\in\mathbf{C}(L,A)_j$, then $f\star g\in\mathbf{C}(L,A)_{i+j}$ is defined by
\begin{equation}	
\begin{split} &(f\star g)(x_1,\ldots,x_{i+j})\\
&\qquad=\sum_{\tau\in S(i,j)}\varepsilon(\tau) (-1)^{|g|(|x_{\tau(1)}|+\cdots+|x_{\tau(i)}|-i)}
f\big(x_{\tau(1)},\ldots,x_{\tau(i)}\big)g\big(x_{\tau(i+1)},\ldots,x_{\tau(i+j)}\big)\\[3pt]
&\qquad =\sum_{\tau\in S_{i+j}}\frac{\varepsilon(\tau)}{i!j!}(-1)^{|g|(|x_{\tau(1)}|+\cdots+|x_{\tau(i)}|-i)}
\,f\big(x_{\tau(1)},\ldots,x_{\tau(i)}\big)g\big(x_{\tau(i+1)},\ldots,x_{\tau(i+j)}\big).\end{split}
	\end{equation}

In particular,  for $a,b\in A=\Hom^*_{\K}\Big(L[1]^{\odot 0}, A\Big)$ and $f\in 
\Hom^{*}_{\K}\Big(L[1]^{\odot 1}, A\Big)$ we have 
\[a\star b=ab,\qquad  a\star f=af,\qquad 
[a,f]_{\star}(x)=[a,f(x)],\]
where $[-,-]_{\star}$ is the graded commutator of $\star$.

On the algebra $\mathbf{C}(L,A)$ we can define the degree one derivations 
\[\delta_0,\delta_1,\delta\in \Hom_{\K}^1(\mathbf{C}(L,A),\mathbf{C}(L,A))\] 
induced by the derivation $d$ on $A$ and by the coderivations $Q_0,Q_1,Q$ on $S(L[1])$ respectively. Namely, given $f\in\mathbf{C}(L,A)$, we put
	\[ \delta_0(f) := df-(-1)^{|f|}fQ_0,\qquad\delta_1(f) := (-1)^{|f|+1}fQ_1,\qquad\delta(f) := \delta_0(f)+\delta_1(f) = df-(-1)^{|f|}fQ.\]
Notice that $\delta_0(\mathbf{C}(L,A)_i)\subset \mathbf{C}(L,A)_i$ and 
$\delta_1(\mathbf{C}(L,A)_i)\subset \mathbf{C}(L,A)_{i+1}$ for every $i$.

Defining a weight gradation in $\mathbf{C}(L,A)$ by setting the elements in $\mathbf{C}(L,A)_i$ of weight $i$ we have that $\delta=\delta_0+\delta_1$ is precisely the weight decomposition of the derivation $\delta$. 	
	
More explicitly, given $f\in\mathbf{C}(L,A)_i$, then  $\delta_0(f)\in\mathbf{C}(L,A)_i$ and $\delta_1(f)\in\mathbf{C}(L,A)_{i+1}$ are defined by: 
\[ \begin{split}
\delta_0(f)(x_1,\ldots, x_i)&= df(x_1,\ldots,x_i)\\
&\quad+(-1)^{|f|}f(\debar x_1,\ldots,x_i)+\cdots+(-1)^{|f|+|x_1|+\cdots+|x_{i-1}|+i-1}f(x_1,\ldots,\debar x_i),
\\[5pt]
\delta_1(f)(x_1,\ldots, x_{i+1})&=(-1)^{|f|+1}\sum_{\tau\in S(2,i-1)} \varepsilon(\tau)(-1)^{|x_{\tau(1)}|} f([x_{\tau(1)},x_{\tau(2)}],x_{\tau(3)},\ldots,x_{\tau(i+1)}).\end{split}\]
	
	Finally, we continue to denote by $R\in\mathbf{C}(L,A)_0$ the degree two element corresponding to the curvature $R\in A$ under the isomorphism 
	\[\mathbf{C}(L,A)_0:=\Hom^*_{\K}\Big(L[1]^{\odot 0},A\Big)=\Hom^*_{\K}(\K,A)= A.\] 
	(in other words, $R(1)=R$ and $R(x_1,\ldots,x_i)=0$ whenever $i>0$).

\begin{proposition} In the above situation, the data $(\mathbf{C}(L,A),\delta,\star,R)$ is a 
curved DG-algebra.
\end{proposition}

\begin{proof}	
	Using the fact that $d$ is an algebra derivation and $Q_0,Q_1,Q$ are coalgebra coderivations, it is easy to check that $\delta_0,\delta_1,\delta$ are algebra derivations with respect to the convolution product $\star$. 
For instance, given $f\in \mathbf{C}(L,A)_i$ and $g\in \mathbf{C}(L,A)_j$  we have 
\[\begin{split} \delta(f\star g)&=d m(f\otimes g)\Delta-(-1)^{|f|+|g|}m(f\otimes g)\Delta Q\\
&=m(d\otimes \Id+\Id\otimes d)(f\otimes g)\Delta-(-1)^{|f|+|g|}m(f\otimes g)(Q\otimes\Id+\Id\otimes Q)\Delta\\
&=m(df\otimes g+(-1)^{|f|}f\otimes dg)\Delta-(-1)^{|f|+|g|}m((-1)^{|g|}fQ\otimes g+f\otimes gQ)\Delta\\
&=m(\delta(f)\otimes g+(-1)^{|f|}f\otimes \delta(g))\Delta=\delta(f)\star g+(-1)^{|f|}f\star\delta(g).
\end{split}	\]

	Moreover, using the fact that $d^2=[R,-]$ and $Q_0^2=Q_1^2=Q^2=0$, one readily checks that 
	\[ \delta(R)=\delta_0(R)=0,\qquad \delta_1^2=\delta_0\delta_1+\delta_1\delta_0=0,\qquad \delta^2=(\delta_0)^2=[R,-]_{\star}.\]
%	
%	In other words, both
%	\[(\mathbf{C}(L,A),\delta_0,\star,R)\quad\mbox{and}\quad(\mathbf{C}(L,A),\delta,\star,R)\]
%	are curved DG-algebras. 
\end{proof}	
	
\begin{definition}\label{def.convolution-algebra}	In the above notation, we shall call $(\mathbf{C}(L,A),\delta,\star,R)$ the \emph{convolution (curved DG) algebra} associated with the curved DG-algebra $A$ and the DG-Lie algebra $L$.
\end{definition}

\begin{definition}\label{def.morfismo-curved}
	A morphism of curved DG-algebras is a morphism of graded algebras that commutes with the derivations and respects the curvatures:
	\[ f\colon (A_1,d_1,R_1)\to (A_2,d_2,R_2),\qquad fd_1=d_2f,\quad f(R_1)=R_2.\]
\end{definition}
\begin{remark} If $f\colon A_1\to A_2$ is a morphism of curved DG-algebras then the induced map 
	$\mathbf{C}(L,A_1)\to \mathbf{C}(L,A_2)$ is a morphism of curved DG-algebras. Similarly, if 
	$M\to L$ is a morphism of DG-Lie algebras (or, more in general, an $L_\infty$ morphism), then 
	the induced map 
	$\mathbf{C}(L,A)\to \mathbf{C}(M,A)$ is a morphism of curved DG-algebras.
\end{remark}

\begin{remark}\label{rem.MC} Given a degree one element $x\in L^1=L[1]^0$, there is an associated morphism of graded associative algebras 
	\begin{eqnarray}\nonumber \operatorname{ev}_x:\mathbf{C}(L,A)&\to& A,%\qquad \operatorname{ev}_x(f)=f(e^{\odot x})
	\\
	\nonumber f\in\mathbf{C}(L,A)_i&\mapsto& \operatorname{ev}_x(f):=\frac{1}{i!}f(x,\ldots,x).
	\end{eqnarray}
	In fact, if $f\in\mathbf{C}(L,A)_i$ and $g\in \mathbf{C}(L,A)_j$, then 
	\[\begin{split} \operatorname{ev}_x\big(f\star g\big)&=\frac{1}{(i+j)!} \big(f\star g\big)(x,\ldots,x) = \frac{1}{(i+j)!}\binom{i+j}{i}f(x,\ldots,x)g(x,\ldots,x)\\
	&=\frac{1}{i!}f(x,\ldots,x)\frac{1}{j!}g(x,\ldots,x)=\operatorname{ev}_x(f)\operatorname{ev}_x(g).
	\end{split}\]
	It is also clear that $\operatorname{ev}_x$ sends the curvature $R\in\mathbf{C}(L,A)_0$ to the curvature $R\in A^2$. In general $\operatorname{ev}_x$ is not  a morphism of curved DG-algebras, but it is so when $x\in\MC(L)$, i.e., when $x$ is a Maurer-Cartan element of $L$. In fact, if $\debar x+[x,x]/2=0$, then for every $f\in\mathbf{C}(L,A)_i$ we have:
	\[\begin{split}  \operatorname{ev}_x\big(\delta(f)\big)&=\frac{1}{i!}\delta_0(f)(x,\ldots,x)+\frac{1}{(i+1)!}\delta_1(f)(x,\ldots,x) \\
	&= \frac{1}{i!} df(x,\ldots,x)+\frac{(-1)^{|f|}}{(i-1)!}f(\debar x,x,\ldots,x)+\frac{(-1)^{|f|}}{2(i-1)!}f([x,x],x,\ldots,x)\\ 
	&= d\operatorname{ev}_x(f).
	\end{split}\]
\end{remark}

For every graded subspace $E\subset A$, we denote by 
$\mathbf{C}(L,E)=\bigoplus_i\Hom^*_{\K}(L[1]^{\odot i},E)\subset \mathbf{C}(L,A)$.

\begin{lemma}\label{lem.bracketinC} 
If $I$ is a curved Lie ideal of $A$, then 
$\mathbf{C}(L,I)$ is a  curved Lie ideal of $\mathbf{C}(L,A)$. Moreover
$\mathbf{C}(L,I)^{(k)}\mathbf{C}(L,A)\subset \mathbf{C}(L,I^{(k)}A)$ for every $k$, and 
$[\mathbf{C}(L,A),\mathbf{C}(L,A)]_{\star}\subset \mathbf{C}(L,[A,A])$.
\end{lemma}

\begin{proof} Immediate from the definitions and from the fact that, 
since the 
unshuffle coproduct  $\Delta$ is graded cocommutative, given 
$f,g\in  \mathbf{C}(L,A)$, every element in the image of $[f,g]_{\star}$ is a linear combination of elements of type
\[\begin{split} 
&m(f\otimes g-(-1)^{|f||g|}g\otimes f)(x\otimes y+(-1)^{(|x|-1)(|y|-1)}y\otimes x)\\
&\qquad =(-1)^{|g|(|x|-1)}[f(x),g(y)]+(-1)^{(|x|+|g|-1)(|y|-1)}[f(y),g(x)].\end{split}\]
\end{proof}

Let $(A,d,R)$ be a curved DG-algebra with a curved Lie ideal $I \subset A$ and 
 denote by 
$\pi \colon A \to A/I$ 
the projection.

Given a DG-Lie algebra $L=(L,\debar,[-,-])$ together with a morphism of graded vector spaces $s\colon L\to A$, the latter can be seen as an element 
of degree $+1$ in $\mathbf{C}(L, A)_1=\Hom^*_{\K}(L[1],A)$, and then it 
gives a sequence of Chern--Simons forms  
\[ W(s)^{k+1} = \frac{1}{k!} \int_{0}^{1} (R + t \delta(s) + t^2 s\star s)^{k}\star s\, dt\ \ \in\mathbf{C}(L,A)^{2k+1}= \Hom^{2k+1}_{\K}(S(L[1]), A),\quad k\ge 0.\]
%and corresponding Chern-Simons classes $\ch(s)_{k+1}^{1}=\tr(W(s)^{k+1})\in \left(\frac{\mathbf{C}(L,A)}{[\mathbf{C}(L,A),\mathbf{C}(L,A)]}\right)^{2k+1}$.
	
Notice that \begin{equation}\label{eq.lin} W(s)^{k+1}=\sum_{i=1}^{2k+1}W(s)^{k+1}_i,\quad\text{with } W(s)^{k+1}_i\in\Hom^{2k+1}_{\K}(L[1]^{\odot i}, A) \mbox{ and } W(s)^{k+1}_1=\frac{R^{k}s}{k!}\,. 
\end{equation}

\begin{lemma}\label{lem.delta-W(x)} In the above situation, let $s\colon L\to A$ be a morphism of graded vector spaces such that the composition  
$\pi s\colon L\to A/I$ is a morphism of DG-Lie algebras. 
Then $\delta (s) + s\star s\in \mathbf{C}(L,I)$ and
\[ \delta W(s)^{k+1} \in [\mathbf{C}(L, A), \mathbf{C}(L, A)] + \mathbf{C}(L, I)^{(k+1)} \quad \forall k \geq 0.\]
Moreover, $\delta (s) + s\star s\in \mathbf{C}(L,I)_1$ if and only if $s$ is a  morphism of 
graded Lie algebras.
\end{lemma}

\begin{proof}
	By Lemma~\ref{lemma.traccia-diff-Wk}
	\[ \delta (W(s)^{k+1}) \treq \frac{1}{(k+1)!} \big( (R + \delta(s) + s\star s)^{k+1} - R^{k+1} \big).\] 
	Since $R\in \mathbf{C}(L, I)$, 
	 it is sufficient to show that 
	$\delta (s) + s\star s$ belongs to $\mathbf{C}(L, I)$. 
	
	Since $\delta_0(s)\in \mathbf{C}(L,A)_1$ and 
	$\delta_1(s), s\star s\in \mathbf{C}(L,A)_2$, the condition 
	$\delta (s) + s\star s\in \mathbf{C}(L, I)$ is equivalent to: 
	\[ \pi \delta_0 s (b_1) =0 ,\quad \pi \delta_1 s (b_1, b_2) + \pi (s \star s)(b_1, b_2)=0,\quad \forall b_1,b_2\in L\,. \]
By definition, $\delta_0 s (b_1)= d s (b_1) - s (\debar b_1)$, so that 
\[\pi \delta_0 s (b_1)= \pi d s (b_1) - \pi s (\debar b_1) = d\pi s (b_1) - \pi s\debar b_1 =0\,.\]

On the other hand, 
\begin{align*} (s \star s)(b_1, b_2) &= m (s \otimes s) (b_1 \otimes b_2 + (-1)^{(|b_1|-1)(|b_2|-1 )} b_2 \otimes b_1) \\
	&= (-1)^{|b_1|-1} s (b_1)s(b_2) + (-1)^{(|b_1|-1)(|b_2|-1 ) + |b_2|-1} s(b_2)s(b_1) \\
	&= (-1)^{|b_1| -1} [s (b_1), s(b_2)].
	\end{align*}
Since $\pi s$ and $\pi$ are morphisms of graded Lie algebras we have
	\[ \pi\delta_1 s (b_1, b_2)= (-1)^{|b_1|} \pi s ([b_1, b_2])=
	(-1)^{|b_1|} [\pi s(b_1),\pi s (b_2)]=(-1)^{|b_1|} \pi [s(b_1),s(b_2)]\,.\] 
The same computation shows that $\delta_1 s+s \star s=0$ if and only if $s$ is a morphism of graded Lie algebras.

\end{proof}

\begin{remark} In fact, the above computations also prove the converse of the first part of Lemma~\ref{lem.delta-W(x)}: namely, that 
	for  a morphism of graded vector spaces $s\colon L\to A$ we have $\delta(s)+s\star s\in \mathbf{C}(L, I)$ if and only if the composition $\pi s$ is a morphism of DG-Lie algebras.
\end{remark}

By Lemma~\ref{lem.bracketinC} we have a natural morphism of differential graded vector spaces
\[ \theta_k\colon \frac{\mathbf{C}(L,A)}{[\mathbf{C}(L,A),\mathbf{C}(L,A)]+\mathbf{C}(L,I)^{(k+1)}\mathbf{C}(L,A)}\to \Hom_{\K}^{*}\left(S(L[1]),\frac{A}{[A,A]+I^{(k+1)}A}\right)\]
and then, for every $s\in \mathbf{C}(L,A)^1$ and every $k\ge 0$, 
we can consider
\[\theta_k(W(s)^{k+1})\in\Hom^{2k+1}_{\K}\left(S(L[1]),\frac{A}{[A,A]+I^{(k+1)}A}\right)\]
and view it as an element in
\[\theta_k(W(s)^{k+1})\in\Hom^0_{\K}\left(S(L[1]),\frac{A}{[A,A]+I^{(k+1)}A}[2k+1]\right).\]

\begin{theorem}\label{th.delta-W(x)}
 In the above situation, suppose that  $\pi s\colon L\to A/I$ is a morphism of DG-Lie algebras. Then $\theta_k(W(s)^{k+1})$ is the corestriction of an $L_{\infty}$ morphism  \[L\rightsquigarrow \dfrac{A}{[A,A]+I^{(k+1)}A}[2k]\] 
with linear Taylor coefficient $\sigma^k_1\pi s$, where $\sigma^k_1$ is the  morphism from Definition \ref{def:semiregmap}, and all Taylor coefficients $L[1]^{\odot i}\to \dfrac{A}{[A,A]+I^{(k+1)}A}[2k+1]$ of degree $i\ge 2k+2$ vanishing. 

\end{theorem}

\begin{proof} Recall that an $L_{\infty}$ morphism $f:L\rightsquigarrow M$ between two DG-Lie algebras is the same as a morphism of DG-coalgebras $F\colon S(L[1])\to S(M[1])$ between their bar constructions. By cofreeness of $S(M[1])$, the correspondence sending $F$ to its corestriction $f=pF$, where we denote by $p\colon S(M[1])\to M[1]$ the natural projection, establishes a bijection between the set of morphisms of graded coalgebras $F\colon S(L[1])\to S(M[1])$ and the set of morphism of graded vector spaces 
$f\colon S(L[1])\to M[1]$: in general, compatibility with the bar differentials translates into a countable sequence of algebraic equations in $f$, see e.g. \cite{yukawate,LMDT}. However, in the particular situation we are concerned with, that is, when the bracket on $M$ is trivial, the situation simplifies considerably, and we have that	$f\in \Hom^0_{\K}(S(L[1]), M[1])$ is the corestriction of an $L_{\infty}$ morphism $F\colon S(L[1])\to S(M[1])$ if and only if $r_1f = fQ$, where we denote by $r_1$ the shifted differential $r_1(m)=-d_M(m)$ on $M[1]$ and by $Q$ the bar differential on $S(L[1])$. In other words, when $M$ has trivial bracket the $L_\infty$ morphisms $L\rightsquigarrow M$ are in bijective correspondence with the set of $0$-cocycles in the  complex
$\Hom^*_{\K}(S(L[1]), M[1])$.

%may be defined as an element $\sigma\in \Hom^0(S(L[1]), M[1])$ such that there exists a morphism of counital DG-coalgebras $F\colon S(L[1])\to  S(M[1])$ satisfying the condition $\sigma=pF$, where $p\colon S(M[1])\to M[1]$ is the natural projection.If such an $F$ exists then it is unique, and the condition of existence is given by a countable sequence of algebraic equations in $\sigma$, see e.g. \cite{yukawate,LMDT}.

On the other hand, by Lemma~\ref{lem.delta-W(x)} the image of $W(s)^{k+1}$ onto $\frac{\mathbf{C}(L,A)}{[\mathbf{C}(L,A),\mathbf{C}(L,A)]+\mathbf{C}(L,I)^{(k+1)}\mathbf{C}(L,A)}$ is a degree $(2k+1)$ cocycle. Therefore, in order to conclude it is sufficient to define the desired $L_{\infty}$ morphism as the image of  
$\theta_k(W(s)^{k+1})$ under the natural isomorphism of differential graded vector spaces
\[ \Hom^*_\K\left(S(L[1]),\frac{A}{[A,A]+I^{(k+1)}A}\right)[2k+1]=
\Hom^*_\K\left(S(L[1]),\frac{A}{[A,A]+I^{(k+1)}A}[2k+1]\right).\]
Finally, the last two statements about the Taylor coefficients follow immediately from the definitions and \eqref{eq.lin}. 
\end{proof}

\begin{corollary}\label{cor.semireg} Given a curved Lie ideal $I$ of a curved DG-algebra $A$, let the symbol $B$ denote the quotient DG-Lie algebra $A/I$. Then, for every morphism of graded vector spaces  
	$s\colon B\to A$ such that $\pi s=\Id_{B}$,  
	the image of $W(s)^{k+1}\in \Hom_{\K}^{2k+1}(S(B[1]),A)$ onto  
	\[ \Hom^{2k+1}_{\K}\left(S(B[1]),\frac{A}{[A,A]+I^{(k+1)}A}\right)=\Hom^0_{\K}\left(S(B[1]),\frac{A}{[A,A]+I^{(k+1)}A}[2k+1]\right),\]
	is  an $L_{\infty}$ morphism  \[\sigma^k=(\sigma^k_1,\sigma^k_2,\ldots,\sigma^k_{2k+1},0,0,\ldots)\colon B\rightsquigarrow \dfrac{A}{[A,A]+I^{(k+1)}A}[2k]\] 
	with all Taylor coefficients of degree $\ge 2k+2$ vanishing and the morphism $\sigma^k_1$  of Definition~\ref{def:semiregmap} as linear Taylor coefficient.\end{corollary}

%%%VECCHIO:

%\red{\begin{corollary}\label{cor.semireg} Let $I$ be a curved Lie ideal of a curved DG-algebra $A$ 
%and denote by $B:=A/I$. For every morphism of graded vector spaces  
%$s\colon B\to A$ such that $\pi s=\Id_{B}$,  
%the image of $W(s)^{k+1}\in \Hom_{\K}^{2k+1}(S(B[1]),A)$ onto  
% \[ \Hom^{2k+1}_{\K}\left(S(B[1]),\frac{A}{[A,A]+I^{(k+1)}A}\right)=\Hom^0_{\K}\left(S(B[1]),\frac{A}{[A,A]+I^{(k+1)}A}[2k+1]\right),\]
%is  an $L_{\infty}$ morphism  \[\sigma^k=(\sigma^k_1,\sigma^k_2,\ldots,\sigma^k_{2k+1},0,0,\ldots)\colon B\rightsquigarrow \dfrac{A}{[A,A]+I^{(k+1)}A}[2k]\] 
%with linear Taylor coefficient $\sigma^k_1$ and all Taylor coefficients of degree $\ge 2k+2$ vanishing.\end{corollary}}

\begin{remark} It follows by Remark~\ref{rem.MC} that the induced push-forward on Maurer--Cartan elements
\begin{eqnarray}\nonumber &&\MC(\sigma^k)\colon\MC(B)\to\MC\left(\frac{A}{[A,A]+I^{(k+1)}A}[2k]\right) = Z^{2k+1}\left(\frac{A}{[A,A]+I^{(k+1)}A}\right),\\\nonumber  
&& \MC(\sigma^k)(x):=\sum_{i=1}^{2k+1}\frac{1}{i!}\sigma^k_i(x,\ldots,x) = \tr\Big(\operatorname{ev}_x\big(W^{k+1}(s)\big)\Big) = \tr\Big(W^{k+1}\big(s(x)\big)\Big),
\end{eqnarray}
sends the Maurer--Cartan element $x\in B^1$ to the residue modulo $\tr\big(I^{(k+1)}A\big)$ of the Chern--Simons class $\cs_{2k+1}\big(s(x)\big)\in A/[A,A]$.
\end{remark}
For general $s$ an explicit combinatorial description of the higher Taylor coefficients  $\sigma^k_{i}$, although possible, 
is quite complicated. In the next section we study these components under the additional assumption that $s\colon B\to A$ is a morphism of 
graded Lie algebras, or equivalently, by Lemma~\ref{lem.delta-W(x)}, that $\delta(s)+s\star s\in \mathbf{C}(B,I)_1$. This is often satisfied in concrete examples (for instance, in the geometric example that we shall consider in the following Section~\ref{sec.connection}).

\bigskip
\section{Explicit formulas in the split case.}
\label{sec.explicit}

Let $\K\langle Z_0,Z_1,Z_2\rangle$ be the associative algebra of noncommutative polynomials in $Z_0,Z_1,Z_2$. Following again \cite[p. 265]{GS}, we denote by $\Sigma[Z_0^p,Z_1^q,Z_2^r]\in \K\langle Z_0,Z_1,Z_2\rangle$, $p,q,r\in \N$, the 
sum of all possible words with $p$ letters $Z_0$, $q$ letters $Z_1$ and $r$ letters $Z_2$.
For instance,  $\Sigma[Z_0^0,Z_1^0,Z_2^0]=1$ and
$\Sigma[Z_0^1,Z_1^2,Z_2^0]=Z_0Z_1^2+Z_1Z_0Z_1+Z_1^2Z_0$.

For every $k\ge 0$ we define homogeneous polynomials 
$V^k(Z_0,Z_1,Z_2)\in \K\langle Z_0,Z_1,Z_2\rangle$
by the formula 
\[ V^k(Z_0,Z_1,Z_2)=\frac{1}{k!} \int_{0}^{1} (Z_0 + tZ_1 + (t^2-t)Z_2)^k\, dt\,.\]

Notice that for every curved DG-algebra $(A,d,R)$ and every $x\in A^1$ we have
\begin{equation}\label{eq.WvsV}
 W(x)^{k+1}=V^k(R,d(x)+x^2,x^2)x\,.
\end{equation}
It is also useful to assign to each variable $Z_i$ the weight $i$, and denote by 
$V^k=\sum_{i=0}^{2k}V^k_i$ the associated isobaric decomposition.
Notice that every monomial in $Z_0,Z_1,Z_2$ of weight $i$ with $r$ occurrences of the variable $Z_2$ (hence $i-2r\ge0$ occurrences of the variable $Z_1$) 
appears in $V^k_i$ with coefficient 
\[ \frac{1}{k!}\int_0^1t^{i-r}(t-1)^rdt=\frac{(-1)^rr!\,(i-r)!}{k!\,(i+1)!}=\frac{(-1)^r}{k!\,(i+1)}\binom{i}{r}^{-1}\,.\]
Therefore, for every $0\le i\le 2k$ 
\[ V^k_i=\sum_{\substack{p+q+r=k\\q+2r=i}}\frac{\,(-1)^rr!\,(i-r)!\,}{k!\,(i+1)!}\;
\Sigma[Z_0^p,Z_1^q,Z_2^r]\,.\]

For instance, one checks that for $0\le i\le 2k\le 6$ the above formula for $V^k_i$ gives:
\[\begin{split} V^0_0&= 1,\qquad V^1_0= Z_0,\qquad V^1_1=\frac{1}{2}Z_1,\qquad V^1_2 = -\frac{1}{6}Z_2,\\[4pt]
V^2_0&= \frac{1}{2}Z_0^2,\quad V^2_1= \frac{1}{4}\big(Z_0Z_1+Z_1Z_0\big),\qquad 
V^2_2=\frac{1}{6}Z_1^2-\frac{1}{12}\big(Z_0Z_2+Z_2Z_0\big),\\[4pt] 
V^2_3&= -\frac{1}{24}\big(Z_1Z_2+Z_2Z_1\big),\qquad V^2_4 = \frac{1}{60}\,Z_2^2,\\[4pt]
V^3_0&= \frac{1}{6}Z_0^3,\quad V^3_1= \frac{1}{12}\big(Z_0^2Z_1+Z_0Z_1Z_0+Z_1Z_0^2\big),\\[4pt] 
V^3_2&=\frac{1}{18}\big(Z_0Z_1^2+Z_1Z_0Z_1+Z_1^2Z_0\big)-\frac{1}{36}\big(Z_0^2Z_2+Z_0Z_2Z_0+Z_2Z^2_0\big),\\[4pt] 
V^3_3&= \frac{1}{24}Z_1^3 -\frac{1}{72}\big(Z_0Z_1Z_2+Z_0Z_2Z_1+Z_1Z_0Z_2+Z_1Z_2Z_0+Z_2Z_0Z_1+Z_2Z_1Z_0\big),\\[4pt] 
V^3_4&= -\frac{1}{120}\big(Z_1^2Z_2+Z_1Z_2Z_1+Z_2Z_1^2\big)+\frac{1}{180}\big(Z_0Z_2^2+Z_2Z_0Z_2+Z_2^2Z_0\big),\\[4pt]
V^3_5&= \frac{1}{360}\big(Z_1Z_2^2+Z_2Z_1Z_2+Z_2^2Z_1\big),\qquad V^3_6 = -\frac{1}{840} Z_2^3.
\end{split}\]

\begin{definition} A \emph{split curved DG-algebra} is the datum of a curved DG-algebra $A=(A,d,R)$ 
equipped with a direct sum decomposition 
$A= B\oplus I$,
where $B\subset A$ is a graded Lie subalgebra and $I\subset A$ is a curved Lie ideal.
\end{definition}

We shall denote by $\imath\colon B\to A$ the inclusion and by 
$P\colon A\to B$ the projection with kernel $I$ (in particular, $\imath,P$ are morphisms of graded Lie algebras) and by $P^\bot:=\operatorname{id}_A-P\colon A\to I$. We shall also denote by
\[ \debar:=Pd\colon B\to B,\qquad \nabla:=P^\bot d\colon B\to I.\]
Notice in particular that since $\operatorname{Ker}(P)=I$ is $d$-closed, then the identity $Pd=PdP$ holds, and in particular
\[\debar^2=(Pd)^2=Pd^2=P[R,-]=0,\]
since $I$ is a Lie  ideal and $R\in I$. Thus $(B,\debar,[-,-])$ is a  DG-Lie algebra and the natural map $(B,\debar)\to (A/I,d)$ is an isomorphism of DG-Lie algebras. 
Moreover,
\begin{equation}\label{equ.dnabla} 
d\nabla+\nabla\debar=[R,-]\colon B\to I,\end{equation}
since for every $b\in B$ we have 
\[ [R,b]=d^2(b)=d\nabla(b)+d\debar(b)=d\nabla(b)+\nabla\debar(b)+\debar\debar(b).\]

\begin{lemma} 
Let $A=B\oplus I$ be a split curved DG-algebra and consider the inclusion $\imath\colon B\hookrightarrow A$ as an element of $\mathbf{C}(B,A)_1$. Then for every $x,y\in B$ we have
\[\begin{split}
&\imath\star \imath\in \Hom^2_{\K}\Big(B[1]^{\odot 2},A\Big)= \mathbf{C}(B,A)^2_2, \qquad \imath\star \imath(x,y) = (-1)^{|x|-1}[x,y],\\
&\delta(\imath)+\imath\star \imath=\nabla\in \Hom^2_{\K}\Big(B[1],I\Big)\subset\mathbf{C}(B,I)^2_1, \qquad (\delta(\imath)+\imath\star \imath)(x) = 
\nabla(x),\\
&R\in \mathbf{C}(B,I)^2_0.\end{split}\]
In particular $W(\imath)^{k+1}=\sum_{i=1}^{2k+1}W(\imath)^{k+1}_i$ with 
\begin{equation*}
%\label{equ.Wcasosplit} 
W(\imath)^{k+1}_i\in \mathbf{C}(B,A)_i,\qquad  W(\imath)^{k+1}_i=V^k_{i-1}(R,\nabla,\imath\star \imath)\star \imath.
\end{equation*}
\end{lemma}

\begin{proof} We have already proved that $\imath\star \imath(x,y) = (-1)^{|x|-1}[x,y]$ in the proof of Lemma~\ref{lem.delta-W(x)}. It remains to show that $(\delta(\imath)+\imath\star \imath)(x) = 
\nabla(x)$.
Again by Lemma~\ref{lem.delta-W(x)} we have $\delta(\imath)+\imath\star \imath=\delta_0(\imath)\in 
\mathbf{C}(B,I)_1$ and then for every $x\in B$ 
\[ (\delta(\imath)+\imath\star \imath)(x)=
\delta_0(\imath)(x)= (d \imath- \imath\debar)(x)=\nabla(x)\,.\]
The last claim follows from Equation \eqref{eq.WvsV}.\end{proof}

\begin{corollary}\label{cor.semiregsplit} 
Let $A=B\oplus I$ be a split curved DG-algebra with inclusion morphism $\imath\colon B\hookrightarrow A$.
For every $i,k$ with $1\le i\le 2k+1$ denote by $\sigma^k_i\in  
\Hom^0_{\K}\left(B[1]^{\odot i},\frac{A}{[A,A]+I^{(k+1)}A}[2k+1]\right)$ the image of
$V^k_{i-1}(R,\nabla,\imath\star \imath)\star \imath$ under the trace map  
\[\mathbf{C}(B,A)^{2k+1}_i\xrightarrow{\tr}\Hom^0_{\K}\left(B[1]^{\odot i},\frac{A}{[A,A]+I^{(k+1)}A}[2k+1]\right).\]
Then  
\[\sigma^k=(\sigma^k_1,\sigma^k_2,\ldots,\sigma^k_{2k+1},0,0,\ldots)\colon 
B\rightsquigarrow \frac{A}{[A,A]+I^{(k+1)}A}[2k]\]
is an $L_{\infty}$ morphism with linear component $\sigma^k_1$.
\end{corollary}

\begin{example} 
	More explicitly, the Taylor coefficients 
	\[ \sigma^k_i:B[1]^{\odot i}\to \frac{A}{[A,A]+I^{k+1}A}[2k+1]\]
	are given on the diagonal, i.e., when all the arguments equal a certain $x\in B^1$, by the formula
	\[ \frac{1}{i!}\sigma^k_i(x,\ldots,x) = \tr\Big(V^k_{i-1}(R,\nabla(x),x^2)x\Big),\]
	and in general is given by the above formula via graded polarization. 
	
	%denoting by \[ F^k_{i-1}=V^k_{i-1}(R,\delta(\imath)+\imath\star\imath,\imath\star \imath)\in 	\Hom^{2k}_{\K}(B[1]^{\odot i-1},A)\]	we have 	\[ \sigma^k_i(x_1,\ldots,x_i)=\sum_{s=1}^i (-1)^{\kappa(s)}\tr(F^k_{i-1}(x_1,\ldots,\widehat{x_s},\ldots,x_i)\, x_s),\]	where $(-1)^{\kappa(s)}$ is the appropriate Koszul sign, viz., 	\[\kappa(s)=\sum_{j\not=s}(|x_j|-1)+\sum_{j>s}(|x_j|-1)(|x_s|-1)\,.\]

	For instance, using the previous explicit formulas for the non-commutative polynomials $V^k_{i-1}$ (together with the cyclic invariance of the trace), we see that for $k\le 3$ the $L_\infty$ morphism from Corollary~\ref{cor.semiregsplit} is given explicitly as follows. For $k=0$, we have the DG-Lie algebra morphism
	\[\sigma^0\colon (B,\debar,[-,-])\to\left(\frac{A}{[A,A]+IA},d,0\right),\qquad \sigma^0(x)=\tr(x).\]
	
	For $k=1$ we have the $L_\infty$ morphism 
	\[ \sigma^1=(\sigma^1_1,\sigma^1_2,\sigma^1_2,0,0,\ldots)\colon (B,\debar,[-,-])\rightsquigarrow \left(\frac{A}{[A,A]+I^{(2)}A}[2],d,0\right)\]
	given by: 
	\[\begin{split} 
	\sigma^1_1(x)&=\tr\Big(Rx\Big),\\
	\sigma^1_2(x_1,x_2)&=\sum_{\tau\in S_2}\frac{\varepsilon(\tau)}{2}\,\tr\Big(\nabla(x_{\tau(1)})x_{\tau(2)}\Big),\\
	\sigma^1_3(x_1,x_2,x_3)&=\sum_{\tau\in S_3}-\frac{\varepsilon(\tau)}{6}\,\tr\Big(x_{\tau(1)}x_{\tau(2)}x_{\tau(3)}\Big),\end{split}\]
	where we denote by $\varepsilon(\tau)$ the symmetric Koszul sign, defined by the identity $x_{\tau(1)}\odot\cdots\odot x_{\tau(i)}=\varepsilon(\tau)\,x_1\odot\cdots\odot x_i$ in the symmetric power $B[1]^{\odot i}$. Hence we recover, in a more general framework,  the formulas of \cite{linfsemireg} with the curvature in place of the Atiyah cocycle, see  Remark~\ref{rem.dipendenza.atiyah} below.

\medskip	
	
For $k=2$ we have the $L_\infty$ morphism 
	\[\sigma^2=(\sigma^2_1,\ldots,\sigma^2_5,0,0,\ldots)\colon(B,\debar[-,-])\rightsquigarrow \bigg(\frac{A}{[A,A]+I^{(3)}A}[4],d,0\bigg)\]
	given by:	
	\[\begin{split} 
	\sigma^2_1(x)&=\frac{1}{2}\,\tr\Big(R^2x\Big),\\
	\sigma^2_2(x_1,x_2)&=\sum_{\tau\in S_2}\frac{\varepsilon(\tau)}{4}\,\tr\Big(R\nabla(x_{\tau(1)})x_{\tau(2)}+\nabla(x_{\tau(1)})Rx_{\tau(2)}\Big),\\
	\sigma^2_3(x_1,x_2,x_3)&=\sum_{\tau\in S_3}\frac{\varepsilon(\tau)}{6}\,\tr\Big( (-1)^{|x_{\tau(1)}|+1}\nabla(x_{\tau(1)})\nabla(x_{\tau(2)})x_{\tau(3)}-Rx_{\tau(1)}x_{\tau(2)}x_{\tau(3)}\Big),\\
\sigma^2_4(x_1,x_2,x_3,x_4)&=\sum_{\tau\in S_4}-\frac{\varepsilon(\tau)}{12}\,\tr\Big(\nabla(x_{\tau(1)})x_{\tau(2)}x_{\tau(3)}x_{\tau(4)}\Big),\\
\sigma^2_5(x_1,x_2,x_3,x_4,x_5)&=\sum_{\tau\in S_5}\frac{\varepsilon(\tau)}{60}\,\tr\Big(x_{\tau(1)}x_{\tau(2)}x_{\tau(3)}x_{\tau(4)}x_{\tau(5)}\Big).\end{split}\]
\medskip	

Finally, for $k=3$ we have the $L_\infty$ morphism  
\[\sigma^3=(\sigma^3_1,\ldots,\sigma^3_7,0,0,\ldots)\colon (B,\debar,[-,-])\rightsquigarrow
\bigg(\frac{A}{[A,A]+I^{(4)}A}[6],d,0\bigg)\]
	given by \[ \sigma^3_i(x_1,\ldots,x_i)=\sum_{\tau\in S_i}\varepsilon(\tau)\tr(P_i(\tau)),\]
where: 
\begin{flalign*}
P_1(\tau)=\frac{1}{6}R^3x_1,&&
\end{flalign*}

\begin{flalign*} &P_2(\tau)=\frac{1}{12}\,\Big(R^2\nabla(x_{\tau(1)})+R\nabla(x_{\tau(1)})R+\nabla(x_{\tau(1)})R^2\Big)x_{\tau(2)}&&
\end{flalign*}

\begin{flalign*}
&P_3(\tau)=-\frac{1}{18}\,R^2x_{\tau(1)}x_{\tau(2)}x_{\tau(3)}-\frac{1}{36}\,Rx_{\tau(1)}x_{\tau(2)}Rx_{\tau(3)}&&\\
&\quad+\frac{1}{18}\,(-1)^{|x_{\tau(1)}|+1}\Big(R\nabla(x_{\tau(1)})\nabla(x_{\tau(2)})+\nabla(x_{\tau(1)})R\nabla(x_{\tau(2)})+\nabla(x_{\tau(1)})\nabla(x_{\tau(2)})R\Big)x_{\tau(3)},&&
\end{flalign*}

\begin{flalign*}
P_4(\tau)&=\frac{1}{24}\,\Big((-1)^{|x_{\tau(2)}|+1}\nabla(x_{\tau(1)})\nabla(x_{\tau(2)})\nabla(x_{\tau(3)})x_{\tau(4)}\Big)&&\\
&\quad-\frac{1}{36}\,\Big(\nabla(x_{\tau(1)})Rx_{\tau(2)}x_{\tau(3)}x_{\tau(4)}+R\nabla(x_{\tau(1)})x_{\tau(2)}x_{\tau(3)}x_{\tau(4)}\Big)&&\\
&\quad -\frac{1}{72}\,\Big((-1)^{|x_{\tau(1)}|+|x_{\tau(2)}|}Rx_{\tau(1)}x_{\tau(2)}\nabla(x_{\tau(3)})x_{\tau(4)}+\nabla(x_{\tau(1)})x_{\tau(2)}x_{\tau(3)}Rx_{\tau(4)}\Big),&&\end{flalign*}

\begin{flalign*}
P_5(\tau)&=\frac{1}{60}\,\Big(Rx_{\tau(1)}x_{\tau(2)}x_{\tau(3)}x_{\tau(4)}x_{\tau(5)}-(-1)^{|x_{\tau(1)}|+1}\nabla(x_{\tau(1)})\nabla(x_{\tau(2)})x_{\tau(3)}x_{\tau(4)}x_{\tau(5)}\Big)&&\\ 
&\quad -\frac{1}{120}\,\Big((-1)^{|x_{\tau(1)}|+|x_{\tau(2)}|+|x_{\tau(3)}|+1}\nabla(x_{\tau(1)})x_{\tau(2)}x_{\tau(3)}\nabla(x_{\tau(4)})x_{\tau(5)}\Big),&&
\end{flalign*}

\begin{flalign*}P_6(\tau)=\frac{1}{120}\,\Big(\nabla(x_{\tau(1)})x_{\tau(2)}x_{\tau(3)}x_{\tau(4)}x_{\tau(5)}x_{\tau(6)}\Big),&&
\end{flalign*}
\begin{flalign*}P_7(\tau)=-\frac{1}{840}\Big(x_{\tau(1)}x_{\tau(2)}x_{\tau(3)}x_{\tau(4)}x_{\tau(5)}x_{\tau(6)}x_{\tau(7)}\Big).&&
\end{flalign*}	

\end{example}

\begin{remark}\label{rem.dipendenza.atiyah} In the above setup, suppose that $I$ is a bilateral associative ideal. Then  
the morphism $\sigma^k_i$ depends only on the class of $R$ in $I/I^{(2)}$ if and only if either $i\le 2$ or $i\ge 2k$.
\end{remark}

\bigskip
\section{Connections of type $(1,0)$ and curved DG-pairs}
\label{sec.connection}

Let $X$ be a complex manifold and let  
\[ \sE^*\colon\qquad 0\to \sE^p\xrightarrow{\,\delta\,}\sE^{p+1}\xrightarrow{\,\delta\,}\cdots\xrightarrow{\,\delta\,}\sE^{q}\to 0,\qquad\quad p,q\in \Z,\quad \delta^2=0,\]
be a fixed finite complex of locally free sheaves of $\Oh_X$-modules. We denote by 
$\HOM^*_{\Oh_X}(\sE^*,\sE^*)$ the graded sheaf of $\Oh_X$-linear endomorphisms of $\sE^*$: 
\[ \HOM^*_{\Oh_X}(\sE^*,\sE^*)=\bigoplus_i \HOM^i_{\Oh_X}(\sE^*,\sE^*),\qquad 
\HOM^i_{\Oh_X}(\sE^*,\sE^*)=\prod_j\HOM_{\Oh_X}(\sE^j,\sE^{i+j}).\]
Then $\HOM^*_{\Oh_X}(\sE^*,\sE^*)$ is a sheaf of locally free DG-Lie algebras over $\Oh_X$, 
with the bracket equal to the graded commutator 
\[ [f,g]=fg-(-1)^{|f|\,|g|}gf\]
and the differential given by 
\[f\mapsto [\delta,f]=\delta f-(-1)^{|f|}f\delta.\]

For every $a,b,r$ 
denote by $\sA^{a,b}_X(\sE^r)\simeq \sA^{a,b}_X\otimes_{\Oh_X}\sE^r$ the sheaf of differential forms of type $(a,b)$ with coefficients in $\sE^r$, and by $\debar\colon \sA^{a,b}_X(\sE^r)\to \sA^{a,b+1}_X(\sE^r)$ the Dolbeault differential.

We consider  
\[ \sA^{*,*}_X(\sE^*)=\bigoplus_{a,b,r}\sA^{a,b}_X(\sE^r)\]
as a sheaf of graded $\sA^{*,*}_X$ modules, where the elements of $\sA^{a,b}_X(\sE^r)$ have degree $a+b+r$.
It is useful to use the dot symbol to denote the natural left multiplication map: 
\[ \sA^{*,*}_X\times \sA^{*,*}_X(\sE^*)\xrightarrow{\,\cdot\,}\sA^{*,*}_X(\sE^*).\]

The differential $\delta$ extends naturally  to a differential 
	\[ \delta\colon \sA^{a,b}_X(\sE^r)\to \sA^{a,b}_X(\sE^{r+1}),\qquad \delta(\phi\cdot e)=
	(-1)^{|\phi|}\,\phi\cdot \delta(e),\qquad \phi\in \sA_X^{a,b},\; e\in \sE^r.\]
	We have that $\debar^2=\delta^2=0$ and $[\debar,\delta]=\debar\delta+\delta\debar=0$, so that
	$\debar+\delta$ is a differential in $\sA^{*,*}_X(\sE^*)$.

In order to fix the notation, we write down the natural extension to locally free complexes of the 
notion of connection.

\begin{definition} Let $\sE^*$ as above: 
 
 \begin{enumerate}
 
\item  A connection  on $\sE^*$ is a $\C$-linear morphism of graded sheaves of degree $+1$
 \[ \nabla\colon \sE^*\to  \sA^{*,*}_X(\sE^*)\]
 such that $\nabla(fe)=d(f)\cdot e+f\cdot \nabla(e)$ for every $f\in \sO_X$, $e\in \sE^*$. Here $d$ denotes the de Rham differential.

\item A connection $\nabla$ as above is called \emph{of type $(1,0)$} if   
$\nabla(e)-\debar(e)\in \oplus_{k}\sA^{1,k}_X(\sE^{i-k})$ for every $i$ and every $e\in \sE^i$.
 \end{enumerate}
Thus, a connection $\nabla$ is of type $(1,0)$ if and only if 
$\nabla=\debar+\sum_{k}\nabla^{1,k}$, with  
$\nabla^{1,k}\colon \sE^i\to \sA^{1,k}_X(\sE^{i-k})$ for every $i$; in this case  
$\nabla^{1,k}$ is $\sO_X$-linear for every $k>0$.
\end{definition}

As in the nongraded case (see e.g. \cite{KoNo}), every connection $\nabla$ extends uniquely to a 
$\C$-linear morphism of graded sheaves of total degree $+1$
 \[ \nabla\colon \sA^{*,*}_X(\sE^*)\to  \sA^{*,*}_X(\sE^*)\]
 such that $\nabla(\phi\cdot \omega)=d(\phi)\cdot \omega+(-1)^{|\phi|}\phi\cdot \nabla(\omega)$, for every 
 $\phi\in \sA^{*,*}_X$, $\omega\in \sA^{*,*}_X(\sE^*)$. 
A connection $\nabla$ as above is called \emph{analytic} if  
$[\nabla,\delta+\debar]=0$; in particular,  a connection $\nabla=\debar+\sum_{k}\nabla^{1,k}$ of type $(1,0)$ is analytic if and only if 
$[\nabla^{1,0}, \delta]=0$ and 
	$[\nabla^{1,k}, \delta] + [\nabla^{1,k-1}, \debar]=0$ for every  $k>0$.

It is clear that  giving a connection $\nabla$ of type $(1,0)$ with $\nabla^{1,k}=0$ for every $k>0$ is the same as giving a connection compatible with the holomorphic structure on every $\sE^i$. In particular, connections of type $(1,0)$ always exist.   

With every connection  $\nabla=\debar+\sum_k\nabla^{1,k}\colon \sA^{*,*}_X(\sE^*)\to  \sA^{*,*}_X(\sE^*)$ of type $(1,0)$  is associated 
a curved DG-pair $(A,I)$ as in Definition~\ref{def.curved-ideal}, according to the following construction.

Denote by  
$A=A^{*,*}_X(\HOM^*_{\Oh_X}(\sE^*,\sE^*))$ the graded associative algebra of global differential forms with values in the graded  sheaf $\HOM^*_{\Oh_X}(\sE^*,\sE^*)$. 
Every element of $A$ may be naturally interpreted  as an  endomorphism of the sheaf $\sA^{*,*}_X(\sE^*)$   and an easy computation  shows that the adjoint operator 
\[ d=[\nabla+\delta,-]\colon A\to A\]
is a well defined derivation. Moreover $R=(\nabla+\delta)^2=\dfrac{1}{2}[\nabla+\delta,\nabla+\delta]$ belongs to $A$ and the triple $(A,d,R)$ is a curved DG-algebra. 

We refer to  
\cite{linfsemireg} for detailed proofs of the above facts. Notice that the definition of  connection of type $(1,0)$  given in  \cite{linfsemireg}  requires  $\nabla^{1,k}=0$ for every $k>0$, but the proofs work the same way since 
the  adjoint operator $[\nabla^{1,k},-]$ is an inner derivation of $A$ for every $k>0$.

Writing $\nabla=\debar+\sum_k\nabla^{1,k}$, since $(\debar+\delta)^2=0$ we have 
\[ R=R_1+R_2,\qquad R_i\in A^{i,*}_X(\HOM^{*}_{\Oh_X}(\sE^*,\sE^*)).\]
In particular, $R$ belongs to the ideal $I=A^{>0,*}_X(\HOM^{*}_{\Oh_X}(\sE^*,\sE^*))$. 
Moreover, $R_1=[\debar+\delta,\sum_k\nabla^{1,k}]$ and then  the connection $\nabla$ is analytic  if and only if $R_{1}=0$.

Since $I^{(k)}=I^{(k)}A=A^{\ge k,*}_X(\HOM^{*}_{\Oh_X}(\sE^*,\sE^*))$ the Atiyah class $\At(A,I)$ of the curved DG-pair is precisely the cohomology class of $R_{1}$ in the complex
$(A^{1,*}_X(\HOM^{*}_{\Oh_X}(\sE^*,\sE^*)),[\debar+\delta,-])$; notice that this complex is the Dolbeault resolution of the complex $\HOM^{*}_{\Oh_X}(\sE^*,\Omega^1_X\otimes\sE^*)$ and therefore
\[ \At(A,I)\in H^2(A^{1,*}_X(\HOM^{*}_{\Oh_X}(\sE^*,\sE^*)))=\EExt^1_X(\sE^*,\Omega^1_X\otimes \sE^*).\] 
Since two connections of type $(1,0)$ differ by a degree 1 element of $A^{1,*}_X(\HOM^{*}_{\Oh_X}(\sE^*,\sE^*))$, the Atiyah class is independent from the choice of the connection. 
Conversely, for every degree 1 element $\phi\in A^{1,*}_X(\HOM^{*}_{\Oh_X}(\sE^*,\sE^*))$ the map $\widehat{\nabla}=\nabla+\phi$ is again a connection of type $(1,0)$ with 1-component of the curvature
$\widehat{R}_1=R_1+[\delta+\debar,\phi]$: it follows that  $\At(A,I)=0$ if and only if $\sE^*$ admits an 
analytic connection of type $(1,0)$.
As already pointed out in \cite{linfsemireg}, the class $\At(A,I)$ is the same as the usual Atiyah class of the complex $\sE^*$ as an object in the derived category of $X$, see also \cite{At} and \cite[Section 10.1]{HL}.

\begin{remark} Let $x\in A$ be such that $[\debar+\delta,x]=0$, then 
$[R_1,x]=[\debar+\delta,[\sum_k \nabla^{1,k},x]]$ and this immediately implies that the Atiyah class is a central element in the cohomology of the differential graded algebra $\operatorname{Gr}_I(A)=\oplus_{k}\dfrac{I^{(k)}}{I^{(k+1)}}$, cf. \cite[Prop. 3.12]{BF}.
\end{remark}

Finally, the same computation of \cite[Lemma 2.6]{linfsemireg} shows that the $A^{*,*}_X$-linear extension of the usual trace map
\[ \Tr\colon  A^{*,*}_X(\HOM^*_{\Oh_X}(\sE^*,\sE^*))\to A^{*,*}_X, \qquad \Tr(\phi\cdot f)=\phi\Tr(f),\]
is a trace map in the sense of Definition~\ref{def.tracemap}. 

\bigskip

Assume now that the complex $\sE^*$ is a resolution of  
a coherent sheaf $\sF$. Then $\At(A,I)$ is equal to the Atiyah class $\At(\sF)$ of $\sF$ and 
the DG-Lie algebra 
\[ \frac{A}{I}=A_X^{0,*}(\HOM^*_{\sO_X}(\sE^*,\sE^*))\]
is precisely the Dolbeault model of the DG-Lie algebra controlling deformations of $\sF$.

Consider the trace map $\Tr\colon A\to A^{*,*}_X$: for every $k\ge 0$ we have 
$\Tr(I^{(k+1)}A)=A^{>k,*}_X$, and then the map $\sigma_1^k$ from Definition \ref{def:semiregmap} becomes
\[ \sigma^k_1\colon A_X^{0,*}(\HOM^*_{\sO_X}(\sE^*,\sE^*))\to A_X^{\le k,*}[2k],
\qquad \sigma_1^k(x)=\frac{1}{k!}\Tr(R_1^kx).\]

Therefore, at the cohomology level $\sigma^k_1$ induces the composition of
\[ \Ext^*_X(\sF,\sF)\to H^{*+k}(X,\Omega^k_X), \qquad x\mapsto \frac{1}{k!}\Tr(\At(\sF)^k\cdot x),\]
with the natural map $j\colon H^{*+k}(X,\Omega^k_X)\to \HH^{2k+*}(A_X^{\le k,*})$.

By Corollary~\ref{cor.main1} the map $\sigma^k_1$ is the linear component of an $L_{\infty}$ morphism, and since the abelian DG-Lie algebra $A_X^{\le k,*}[2k]$ has trivial obstructions we immediately obtain the following result.

\begin{corollary}\label{cor.main2} Let $\sF$ be a coherent sheaf on a complex manifold $X$ admitting a locally free resolution. 
Then for every $k\ge 0$ the semiregularity map 
\[ \Ext^2_X(\sF,\sF)\to \HH^{2k+2}(A_X^{\le k,*}),\qquad x\mapsto \frac{1}{k!}\Tr(\At(\sF)^k\cdot x),\]
annihilates obstructions to deformations of $\sF$.
\end{corollary}

\begin{proof} Since $X$ is assumed smooth, by Hilbert's syzygy theorem, if $\sF$ admits a locally free resolution, then it also admits a finite locally free resolution.
\end{proof}

For $k=0$ Corollary~\ref{cor.main2} is a classical result by Mukai and Artamkin \cite{Arta}, while the case $k=1$ was proved  in \cite{linfsemireg} by constructing explicitly an  $L_\infty$ lifting of $\sigma^1_1$.

\begin{corollary}\label{cor.main3} Let $\sF$ be a coherent sheaf on a complex projective manifold  $X$. 
Then for every $k\ge 0$ the semiregularity map 
\[ \Ext^2_X(\sF,\sF)\to H^{k+2}(X,\Omega^{k}_X),\qquad x\mapsto \frac{1}{k!}\Tr(\At(\sF)^k\cdot x),\]
annihilates obstructions to deformations of $\sF$.
\end{corollary}

\begin{proof} Since $X$ is projective every coherent sheaf admits a locally free resolution. 
Moreover, the Hodge to de Rham spectral sequence degenerates at $E_1$ and therefore the natural map 
$H^{k+2}(X,\Omega^{k}_X)\to \HH^{2k+2}(A_X^{\le k,*})$ is injective.\end{proof}

\subsection{Acknowledgements.} We thank Victor Ginzburg for pointing out to our attention Quillen's paper \cite{Quillen88}.  This work was carried out in the framework of the PRIN project  \lq\lq Moduli  and Lie theory\rq\rq\   2017YRA3LK.

\end{document}